\documentclass[pdflatex,sn-basic]{sn-jnl}

\usepackage{array}

\jyear{2021}%

\theoremstyle{thmstyleone}%
\newtheorem{theorem}{Theorem}
\newtheorem{proposition}[theorem]{Proposition}%

\theoremstyle{thmstyletwo}%

\theoremstyle{thmstylethree}%

\newcommand{\RR}{\mathbb{R}}
\newcommand{\IR}{\mathbb{IR}}
\newcommand{\lb}[1]{\underline{#1}}
\newcommand{\ub}[1]{\overline{#1}}
\newcommand{\ib}[1]{\mathbf{#1}}
\newcommand{\ubf}{\overline{f}_{\textnormal{fin}}}
\def\F{\mathcal{F}}

\begin{document}

\title[Interval TP: Feasibility, Optimality and the Worst Optimal Value]{Interval Transportation Problem: Feasibility, Optimality and the Worst Optimal Value}


\author*[1,2]{\fnm{Elif} \sur{Garajová}}\email{elif@kam.mff.cuni.cz}

\author[2,3]{\fnm{Miroslav} \sur{Rada}}\email{miroslav.rada@vse.cz}


\affil*[1]{\orgdiv{Dpt. of Applied Mathematics}, \orgname{Charles University, Faculty of Mathematics and Physics}, \orgaddress{\street{Malostransk\'{e} n\'{a}m. 25}, \city{Prague}, \postcode{11800}, \country{Czech Republic}}}

\affil[2]{\orgdiv{Dpt. of Econometrics}, \orgname{Prague University of Economics and Business, Faculty of Informatics and Statistics}, \orgaddress{\street{n\'{a}m.~W.~Churchilla 4}, \city{Prague}, \postcode{13067}, \country{Czech Republic}}}

\affil[3]{\orgdiv{Dpt. of Financial Accounting and Auditing}, \orgname{Prague University of Economics and Business, Faculty of Finance and Accounting}, \orgaddress{\street{n\'{a}m.~W.~Churchilla 4}, \city{Prague}, \postcode{13067}, \country{Czech Republic}}}


\abstract{We consider the model of a transportation problem with the objective of finding a minimum-cost transportation plan for shipping a~given commodity from a~set of supply centers to the customers. Since the exact values of supply and demand and the exact transportation costs are not always available for real-world problems, we adopt the approach of interval programming to represent such uncertainty, resulting in the model of an interval transportation problem. The interval model assumes that lower and upper bounds on the data are given and the values can be independently perturbed within these bounds.

In this paper, we provide an overview of conditions for checking basic properties of the interval transportation problems commonly studied in interval programming, such as weak and strong feasibility or optimality. We derive a condition for testing weak optimality of a solution in polynomial time by finding a suitable scenario of the problem. Further, we formulate a similar condition for testing strong optimality of a solution for transportation problems with interval supply and demand (and exact costs).

Moreover, we also survey the results on computing the best and the worst optimal value. We build on an exact method for solving the NP-hard problem of computing the worst (finite) optimal value of the interval transportation problem based on a decomposition of the optimal solution set by complementary slackness. Finally, we conduct computational experiments to show that the method can be competitive with the state-of-the-art heuristic algorithms.
}




\maketitle

\section{Introduction}\label{sec:intro}
Transportation problems~\citep{Hitchcock:1941:TP} often arise in various practical applications, and as such they belong to the commonly studied and used models in operations research. The goal of a transportation problem is to find a~transportation plan for shipping a given commodity from a~set of supply centers to the customers, while minimizing the total transportation costs and respecting the available supply levels and the customer demand.

Since uncertainty is an important factor present in many real-world problems, we address the issue that the values of supply and demand and even the unit transportation costs may not always be known exactly in advance. In this paper, we formulate the transportation problem as a model of interval linear programming~(see e.g. \cite{Rohn:2006:Intervallinearprogramminga} for general introduction to interval linear programming), where we assume that only lower and upper bounds on the exact quantities are given and the values can be perturbed independently within these bounds. Such uncertain model is known in the literature as the \emph{interval transportation problem}~\citep{Chanas:1993:ITP}.

One of the essential tasks in interval programming is computing the optimal value range, i.e. the best and the worst value that is optimal for some choice of the uncertain coefficients. In the traditional sense, these values are allowed to be infinite, if there is a scenario of the interval program that is unbounded or infeasible. However, for some models that often contain an infeasible scenario, it may be more desirable to compute the \emph{worst finite optimal value}~\citep{Hladik:2018:FinWorstCase}. This is also the case for interval transportation problems -- if we require both supply and demand to be satisfied exactly, then any slight change in one of the values will render a feasible scenario infeasible.

Computing the worst finite optimal value is a challenging task in interval linear programming, which was also recently proved to be NP-hard for a special class of interval transportation problems~\citep{Hoppmann:2021:MinCostFlow}. Several algorithms providing an approximate solution for the problem have been proposed throughout the years, such as the heuristic method by \cite{Juman:2014:ITP}, the genetic algorithm by~\cite{Xie:2017:ITP} and the iterative local search by \cite{Cerulli:2017:BestWorstValues}. Moreover, \cite{Liu:2003:ITP} derived a linearly constrained nonlinear formulation by duality, which provides an exact solution in theory, but proved to be too difficult for the solvers to obtain the global optimum in practice, even for small problems. Stronger properties were derived for interval transportation problems with costs immune against the transportation paradox~\citep{Szwarc:1971:Paradox}, allowing the design of more powerful methods for this special case \citep{DAmbrosio:2020:optimalvaluerange,Carrabs:2021:IITP}.

In the first part of the paper, we discuss several properties often considered in interval linear programming, such as \emph{weak or strong feasibility} and optimality of a solution or of the interval problem itself. Characterizations of these properties for general interval linear programs were studied by \cite{Rohn:2006:Intervallinearprogramminga,Hladik:2012a}, further results for programs with a fixed constraint matrix were presented by \cite{Garajova:2017:OnTheProp}. Here, we derive analogous characterizations for interval transportation problems and study the computational complexity of testing these properties. Namely, we prove a condition for testing \emph{weak optimality} of a given solution in polynomial time by finding a suitable scenario of the transportation problem. We also formulate a polynomial-time condition for testing \emph{strong optimality} of a solution for transportation problems with interval supply and demand vectors and exact costs.

We also study the problem of computing the best and the worst (finite) optimal value of an interval transportation problem. We build on an \emph{exact method} for solving the NP-hard problem of computing the worst finite optimal value, which is based on the characterization of the optimal solution set via duality and complementary slackness~\citep{Garajova:2019:optimalsolutionseta,garajova:2021:ExactMethodWorst}. A similar approach using guided basis enumeration was also introduced by \cite{Hladik:2018:FinWorstCase}. Further, we conduct a series of computational experiments indicating that the proposed exact method can be competitive with the state-of-the-art heuristic algorithms.

The paper is organized as follows: First, we formally introduce the model of an interval transportation problem and define the optimal solutions and optimal values in Section~\ref{sec:itp}. In Section~\ref{sec:complexity}, we provide an overview of conditions for checking different properties of interval transportation problems and study their computational complexity. In Section~\ref{sec:worstFinVal}, we derive the exact method for computing the worst finite optimal value. Section~\ref{sec:experiment} presents the results of numerical experiments testing the proposed method against other available algorithms for solving the problem. 

\section{Interval Transportation Problem}\label{sec:itp}
Let us now introduce the interval programming model of a transportation problem with uncertain supply, demand and cost coefficients.

\subsection{Interval data} Here, we use the symbol $\IR$ to denote the set of all closed real intervals. Given two real vectors $\lb{a}, \ub{a} \in \RR^{n}$ satisfying $\lb{a} \le \ub{a}$ (element-wise), we define an \emph{interval vector} $\ib{a} \in \IR^{n}$ as the set of vectors
\[ \ib{a} = [\lb{a}, \ub{a}] = \{a \in \RR^{n} : \lb{a} \le a \le \ub{a} \}, \]
where $\lb{a}$ and $\ub{a}$ are the \emph{lower and upper bound} of the interval vector $\ib{a}$, respectively. An \emph{interval matrix} can be defined analogously. Throughout the paper, intervals and interval objects are denoted by bold letters (or written out using their lower and upper bounds).

\subsection{Problem formulation}
Suppose we are given the following data:
\begin{itemize}
	\item a set of $m$ \emph{sources} (or supply centers) denoted by $I = \{1, \dots, m\}$, such that each source $i \in I$ has a limited \emph{supply} $s_i$,
	\item a set of $n$ \emph{destinations} (or customers) denoted by $J = \{1, \dots, n\}$, such that each destination $j \in J$ has a \emph{demand} $d_j$ to be satisfied,
	\item a \emph{unit cost} $c_{ij}$ of transporting one unit of goods from source~$i$ to destination~$j$.
\end{itemize}
Furthermore, assume that the supply, demand and costs are uncertain quantities that can vary within the given non-negative intervals $\ib{s}_i = [\lb{s}_i, \ub{s}_i]$, $\ib{d}_j = [\lb{d}_j, \ub{d}_j]$ and $\ib{c}_{ij} = [\lb{c}_{ij}, \ub{c}_{ij}]$, respectively.
The objective of the transportation problem is to find a minimum-cost transportation plan for shipping goods from the sources to the destinations such that the supply and demand requirements are satisfied.

Formally, given the interval supply and demand vectors $\ib{s} \in \IR^m$, $\ib{d} \in \IR^n$ and the transportation costs $\ib{c} \in \IR^{m\times n}$, the \emph{interval transportation problem (ITP)} can be represented by an interval linear programming model, which is understood as the set of all linear programs (transportation problems) with the costs, supply and demand vectors lying in the corresponding intervals $\ib{c}$, $\ib{s}$ and $\ib{d}$. A particular linear program in the interval transportation problem, which is determined by a cost $c \in \ib{c}$, supply vector $s \in \ib{s}$ and demand vector $d \in \ib{d}$, is called a~\emph{scenario} of the ITP. 

In the recent literature, an interval transportation problem in the following form is usually considered:
\begin{equation}\tag{ITP$^\le$}\label{eq:itp:unbalanced}
	\begin{alignedat}{3}
		\text{minimize } 	& \quad & \omit\rlap{$\displaystyle \sum_{i \in I} \sum_{j \in J} [\lb{c}_{ij}, \ub{c}_{ij}] x_{ij}$} \\
		\text{subject to } 	&	& \sum_{j \in J} x_{ij} & \le [\lb{s}_i, \ub{s}_i], & \qquad & \forall i \in I,  \\
		&	& \sum_{i \in I} x_{ij} & = [\lb{d}_j, \ub{d}_j], 	&		& \forall j \in J, \\
		&	& x_{ij} 				& \ge 0, 					&		& \forall i \in I, j \in J,
	\end{alignedat}
\end{equation}
where the variables $x_{ij}$ correspond to the amount of goods transported from source $i$ to destination $j$. 
If the supplies are required to be completely depleted, we can also consider the commonly used equation-constrained formulation
\begin{equation}\tag{ITP$^=$}\label{eq:itp:balanced}
	\begin{alignedat}{3}
		\text{minimize } 	& \quad & \omit\rlap{$\displaystyle \sum_{i \in I} \sum_{j \in J} [\lb{c}_{ij}, \ub{c}_{ij}] x_{ij}$} \\
		\text{subject to } 	&	& \sum_{j \in J} x_{ij} & = [\lb{s}_i, \ub{s}_i], & \qquad & \forall i \in I,  \\
		&	& \sum_{i \in I} x_{ij} & = [\lb{d}_j, \ub{d}_j], 	&		& \forall j \in J, \\
		&	& x_{ij} 				& \ge 0, 					&		& \forall i \in I, j \in J.
	\end{alignedat}
\end{equation}

\subsection{Feasibility and optimality}\label{ssec:feasopt} An interval linear program comprises an entire set of classical linear programs and there are several different notions used to define a feasible or an optimal solution. Perhaps the most natural is the notion of weak and strong solutions.

A given solution $x \in \RR^{m \times n}$ is called \emph{weakly feasible/optimal} if it is  feasible/optimal for at least one scenario $(c,s,d) \in (\ib{c}, \ib{s}, \ib{d})$ of the interval transportation problem. Similarly, a given $x \in \RR^{m \times n}$ is called \emph{strongly feasible/optimal}, if it is feasible/optimal for each scenario of the transportation problem. 

Let us denote by $\mathcal{F}(s,d)$ the feasible set of a scenario of the ITP with supply~$s$ and demand~$d$. Then, we can also describe the set of all weakly and strongly feasible solutions as the union $ \bigcup \{ \mathcal{F}(s,d) : s \in \ib{s},\, d \in \ib{d}\} $ and the intersection $ \bigcap \{ \mathcal{F}(s,d) : s \in \ib{s},\, d \in \ib{d}\} $ of all feasible sets, respectively.

Apart from checking the feasibility or optimality properties of a given solution, we can also consider properties of the interval transportation problem itself. We say that the ITP is \emph{weakly feasible/optimal} if there is at least one scenario that possesses a feasible/optimal solution. Analogously, it is said to be \emph{strongly feasible/optimal}, if each scenario of the ITP has a feasible/optimal solution. 

Regarding optimal values, we can consider the \emph{best optimal value} $\lb{f}$ and the \emph{worst optimal value} $\ub{f}$, which are defined as
\begin{align*}
	&\lb{f} = \inf\  \{ f(c,s,d) : c \in \ib{c},\, s \in \ib{s},\, d \in \ib{d} \},\\
	&\ub{f} = \sup\, \{ f(c,s,d) : c \in \ib{c},\, s \in \ib{s},\, d \in \ib{d} \},
\end{align*}
where $f(c,s,d)$ denotes the optimal value of a given scenario $(c,s,d)$, with $f(c,s,d) = \infty$ if the scenario is infeasible.
Since the objective value of the transportation problem with non-negative costs is bounded below, the best optimal value will be finite (provided there is at least one feasible solution). However, the worst optimal value can be $\ub{f} = \infty$, if the interval program contains an infeasible scenario. Therefore, we will mainly focus on computing the worst finite optimal value, which is defined as the worst optimal value over all feasible scenarios, i.e.
\[ \ubf = \max\, \{ f(c,s,d) : c \in \ib{c}, s \in \ib{s}, d \in \ib{d} \text{ with } \F(s,d) \neq \emptyset \}. \]
For the transportation problem, we can also express the condition $\F(s,d) \neq \emptyset$ in terms of supply and demand, namely
\begin{align*}
	\sum_{i \in I} s_i \ge \sum_{j\in J} d_j & \text{ for the formulation } \eqref{eq:itp:unbalanced},\\
	\sum_{i \in I} s_i = \sum_{j\in J} d_j  & \text{ for the formulation } \eqref{eq:itp:balanced}.
\end{align*}
Scenarios satisfying the latter condition (with equal total supply and demand) are also referred to as \emph{balanced} transportation problems.

Note that while it is possible to transform a mixed transportation problem with fixed data to an equation-constrained one, the usual transformations are not always applicable to interval transportation problems, since they may change some properties of the interval program. For example, by adding a~dummy node to transform the~\eqref{eq:itp:unbalanced} formulation to the~\eqref{eq:itp:balanced} formulation, we may change strong feasibility or strong optimality of the problem, since this transformation can introduce an infeasible scenario $(s,d)$ with $\sum_{i \in I} s_i \neq \sum_{j \in J} d_j$ (and thus it can also change $\ub{f}$). 

\section{Properties of Interval Transportation Problems}\label{sec:complexity}

In this section, we discuss and survey the computational complexity of testing various properties of interval transportation problems and derive the corresponding testing conditions. We consider the traditionally studied problems in interval linear programming (see also \cite{Hladik:2012a}) -- weak and strong feasibility, optimality and the problem of computing the optimal value range. 

\begin{table}[b]
	\caption{Computational complexity of testing properties of interval linear programs with a fixed coefficient matrix and interval transportation problems. First four rows of the table refer to the respective properties of the interval program itself, the last four rows refer to properties of a particular solution (see Section~\ref{ssec:feasopt}.)}\label{tab:prop:itp}
	\begin{tabular*}{\textwidth}{l @{\extracolsep{\fill}} ccc}
		\hline\noalign{\smallskip}
		& $\min\, \ib{c}^T x$ & Interval transportation \\
		& $Ax \le \ib{a}, Bx = \ib{b},\, x \ge 0$ & problem \\ 
		\noalign{\smallskip}\hline\noalign{\smallskip}
		strong feasibility & co-NP-hard & {polynomial}\\
		weak feasibility & polynomial & polynomial\\ \hline
		strong optimality & co-NP-hard & {polynomial}\\
		weak optimality & polynomial & polynomial \\ \hline
		best optimal value & polynomial & polynomial \\
		worst optimal value & NP-hard & {polynomial} \\
		worst finite optimal value & NP-hard & ?\footnotemark[1] \\ \hline
		strong feasibility (solution) & polynomial & polynomial \\
		weak feasibility (solution) & polynomial & polynomial\\ \hline
		strong optimality (solution) & co-NP-hard & ? \\
		weak optimality (solution) & ? & polynomial \\
		\noalign{\smallskip}\hline\noalign{\smallskip}
	\end{tabular*}
	\footnotetext[1]{The worst finite optimal value is NP-hard for \eqref{eq:itp:balanced}.}
\end{table}

An overview of the complexity of testing these properties for general linear programs with interval objective and right-hand side is presented in Table~\ref{tab:prop:itp} (see also \cite{Garajova:2017:OnTheProp} for further details). Note that the first four rows of the table refer to the properties of the interval program itself, while the last four rows refer to the properties of a particular solution (as they were introduced in Section~\ref{ssec:feasopt}).

The results in this section focus mainly on the properties of \eqref{eq:itp:unbalanced}, however, the results for the alternative formulation \eqref{eq:itp:balanced} are very similar and can be derived in an analogous manner. The computational complexity of testing the properties of interval transportation problems is also summarized in Table~\ref{tab:prop:itp}.

\subsection{Weak and Strong Feasibility}
We introduced the weakly and strongly feasible solutions of an interval transportation problem in Section~\ref{sec:itp}. Checking weak feasibility of a given solution can be done in polynomial time even for general interval linear systems (see \cite{Oettli:1964:Compatibilityapproximatesolution} and \cite{Gerlach:1981:ZurLosunglinearer}). For interval transportation problems, the characterization can be simplified to the condition presented in Proposition~\ref{prop:weakFeas:sol}. 

Note that the result can also be directly derived from the fact that the set of all weakly feasible solutions is the union of the feasible sets $\mathcal{F}(s,d)$ over all scenarios. Since $s^1 \le s^2$ implies $\mathcal{F}(s^1,d) \subseteq \mathcal{F}(s^2,d)$ for any supply vectors $s^1, s^2 \in \ib{s}$, the set of weakly feasible solutions is determined by $\ub{s}$.

\begin{proposition}\label{prop:weakFeas:sol}
A given $x \in \RR^{m \times n}$ is a weakly feasible solution of~\eqref{eq:itp:unbalanced} if and only if it solves the linear system
\begin{equation}\label{eq:prop:weakFeas}
		\begin{alignedat}{3}
			&	& \sum_{j \in J} x_{ij} & \le \ub{s}_i, & \qquad & \forall i \in I,  \\
			&	\lb{d}_j \le&\ \sum_{i \in I} x_{ij} & \le \ub{d}_j, 	&		& \forall j \in J, \\
			&	& x_{ij} 				& \ge 0, 					&		& \forall i \in I, j \in J.
		\end{alignedat}
\end{equation}
\end{proposition}

Strong feasibility of a given solution can also be checked in polynomial time for general interval linear programs \citep{Rohn:2006:Intervallinearprogramminga}. Since in both ITP formulations, the demands have to be satisfied exactly, a strongly feasible solution is rather uncommon. Again, the characterization shown in Proposition~\ref{prop:strFeas:sol} follows from the general result. 

Using the fact that the set of all strongly feasible solutions is the intersection of the feasible sets $\mathcal{F}(s,d)$, it is easy to see that the set of strongly feasible solutions is determined by the lowest possible supply $\lb{s}$. Furthermore, no strongly feasible solution can exist if $\lb{d} \neq \ub{d}$, since the demand has to be satisfied exactly in all scenarios.

\begin{proposition}\label{prop:strFeas:sol}
A given $x \in \RR^{m \times n}$ is a strongly feasible solution of \eqref{eq:itp:unbalanced} if and only if $\lb{d} = \ub{d}$ and $x$ solves the linear system
\begin{equation}\label{eq:prop:strFeas}
	\begin{alignedat}{3}
		&	& \sum_{j \in J} x_{ij} & \le \lb{s}_i, & \qquad & \forall i \in I,  \\
		&	& \sum_{i \in I} x_{ij} & = \ub{d}_j, 	&		& \forall j \in J, \\
		&	& x_{ij} 				& \ge 0, 					&		& \forall i \in I, j \in J.
	\end{alignedat}
\end{equation}
\end{proposition}

Weak and strong feasibility of the interval transportation problem itself can also be tested easily, as it only depends on whether it is possible to satisfy the required demands in the best/worst scenario.

\begin{proposition}
	Problem \eqref{eq:itp:unbalanced} is weakly feasible if and only if
	\[ 
		\sum_{i \in I} \ub{s}_i \ge \sum_{j\in J} \lb{d}_j.
	\]
	Furthermore, problem \eqref{eq:itp:unbalanced} is strongly feasible if and only if
	\[ 
		\sum_{i \in I} \lb{s}_i \ge \sum_{j\in J} \ub{d}_j.
	\]
\end{proposition}

Note that this is significantly different from the general interval programs, for which strong feasibility can be challenging to decide (see Table~\ref{tab:prop:itp}).

\subsection{Weak and Strong Optimality}

Weak and strong optimality of a given solution was also introduced in Section~\ref{sec:itp}. Checking weak optimality can be performed by using the characterization of optimal solutions known from the theory of linear programming: an optimal solution needs to be feasible for the primal program and there needs to be a feasible solution of the dual program such that the primal and the dual objective values for these solutions are equal (for some scenario with $c \in \ib{c}$, $s \in \ib{s}$, $d \in \ib{d}$).

Considering a scenario $(c,s,d) \in (\ib{c}, \ib{s}, \ib{d})$ of the \eqref{eq:itp:unbalanced} model, the dual program reads
\begin{equation}\label{eq:itp:dual}
	\begin{alignedat}{3}
		\text{maximize } 	& \quad & \omit\rlap{$\displaystyle \sum_{i \in I}\nolimits s_i u_i + \sum_{j \in J}\nolimits d_j v_j$} \\
		\text{subject to } 	&	& u_i + v_j & \le c_{ij}, & \qquad\quad & \forall i \in I, j \in J  \\
		&	& u_{i} 				& \le 0, 					&		& \forall i \in I.
	\end{alignedat}
\end{equation}
A solution $x \in \mathbb{R}^{m \times n}$ is optimal for the scenario $(c,s,d)$ if it is feasible and there exists a solution $(u, v) \in \mathbb{R}^{m+n}$, such that $(u,v)$ is feasible for the dual program \eqref{eq:itp:dual} and the primal and dual objectives are equal (also known as the property of strong duality), i.e.
\begin{equation}\label{eq:itp:strongdual}
	\sum_{i\in I}\nolimits \sum_{j \in J}\nolimits c_{ij}x_{ij} = \sum_{i \in I}\nolimits s_i u_i + \sum_{j \in J}\nolimits d_j v_j.
\end{equation}

A suitable scenario for checking weak optimality of a given solution $x$ can be found as shown in Proposition~\ref{prop:weakOpt:sol}. Here, the values of supply and demand are chosen depending on $x$, so that the demand is satisfied exactly and there is sufficient supply available, while respecting the lower bound $\lb{s}$. Suitable costs~$c_{ij}$ and values for the dual variables $u_i$ and $v_j$ are then computed within the linear system.
\begin{proposition}\label{prop:weakOpt:sol}
	A given $x \in \RR^{m \times n}$ is a weakly optimal solution of \eqref{eq:itp:unbalanced} if and only if it is weakly feasible and 
	the following linear system with variables $u_i$, $v_j$ and $c_{ij}$ is feasible:
	\begin{equation}\label{eq:prop:weakOpt:sol}
		\begin{alignedat}{3}
			\sum_{i\in I}\nolimits \sum_{j \in J}\nolimits c_{ij}x_{ij} &= \sum_{i \in I}\nolimits s_i u_i + \sum_{j \in J}\nolimits d_j v_j,\\
			 u_i + v_j &\le {c}_{ij}, && \forall i \in I, j \in J, \\
			 u_i &\le 0, && \forall i \in I,\\
			 \lb{c}_{ij} \le c_{ij} &\le \ub{c}_{ij}, & \quad &\forall i \in I, j \in J,
		\end{alignedat}
	\end{equation}
	where $s_{i} = \max\{\lb{s}_i, \sum_{j \in J} x_{ij}\}$ and $d_j = \sum_{i \in I} x_{ij}$.
\end{proposition}
\begin{proof}

First, let us assume that $x$ is a weakly feasible solution and system~\eqref{eq:prop:weakOpt:sol} is feasible.
Then the scenario $(s,d)$ defined in the statement is a valid scenario of \eqref{eq:itp:unbalanced} with $s\in\ib{s}$, $d \in \ib{d}$ and clearly $x \in \mathcal{F}(s,d)$. 
Moreover, the feasibility of system~\eqref{eq:prop:weakOpt:sol} ensures that there exists $c\in \ib{c}$ such that dual feasibility and equality of objective values are satisfied.  Thus $x$ is optimal for the scenario $(c,s,d)$ and therefore weakly optimal.

On the other hand, suppose that $x$ is weakly optimal for the given instance of~\eqref{eq:itp:unbalanced}. 
Then, $x$ is optimal for some scenario $(c', s', d')$. By defining $s_{i}$ and $d_j$ for all $i \in I, j \in J$ as in the statement of the proposition, we obtain a valid scenario of the given~\eqref{eq:itp:unbalanced} with $s \in \ib{s}$ and $d \in \ib{d}$, since \[d_j = \sum_{i \in I} x_{ij} = d'_j \in \ib{d}_j, \quad \text{ and, }\quad \lb{s}_i \le s_i = \max\left\{\lb{s}_i, \sum_{j \in J} x_{ij}\right\} \le s'_i \le \ub{s}_i.\]
Furthermore, by the definition of vectors $s$ and $d$, we also have $x \in \mathcal{F}(s,d)$.

Now we prove that $x$ is also optimal for the scenario $(c', s, d)$. Suppose for contradiction that $x$ is not optimal for the new supply values $s$, then there is a solution~$x^*$ with a better objective value
\[ 
	\sum_{i \in I} \sum_{j \in J} c'_{ij} x^*_{ij} < \sum_{i \in I} \sum_{j \in J} c'_{ij} x_{ij}.
\]
Such $x^*$ is then also feasible for the scenario with supply $s'$ with the same objective value, which contradicts the assumption that $x$ is optimal for $s'$. Therefore, a weakly optimal solution $x$ is also optimal for the scenario chosen in the statement of the proposition. Since the constraints of system~\eqref{eq:prop:weakOpt:sol} describe the optimality conditions for $x$ in the considered scenario (apart from primal optimality), namely the dual feasibility constraints~\eqref{eq:itp:dual} and strong duality~\eqref{eq:itp:strongdual}, we conclude that the system is feasible.
\end{proof}
Note that the values $x_{ij}$ (and thus also $s_i$ and $d_j$) are fixed in system \eqref{eq:prop:weakOpt:sol}, so it is indeed a linear system and the property can be checked in polynomial time by means of linear programming. If the costs are not affected by uncertainty, we can also formulate a linear program to check strong optimality of a solution.

\begin{proposition}
	If $\lb{c} = \ub{c}$, then a given $x \in \RR^{m \times n}$ is a strongly optimal solution of \eqref{eq:itp:unbalanced} if and only if it is strongly feasible and also optimal for the linear program
	\begin{equation}\label{eq:strOpt:c}
		\begin{alignedat}{3}
			\textnormal{minimize } 	& \quad & \omit\rlap{$\displaystyle \sum_{i \in I} \sum_{j \in J} \ub{c}_{ij} x_{ij}$} \\
			\textnormal{subject to } 	&	& \sum_{j \in J} x_{ij} & \le \ub{s}_i, & \qquad & \forall i \in I,  \\
			&	& \sum_{i \in I} x_{ij} & = \ub{d}_j, 	&		& \forall j \in J, \\
			&	& x_{ij} 				& \ge 0, 					&		& \forall i \in I, j \in J.
		\end{alignedat}
	\end{equation}
\end{proposition}
\begin{proof}
	Obviously, if $x$ is strongly optimal, then it is also strongly feasible and optimal for the scenario $(\ub{c}, \ub{s}, \ub{d})$. 
	
	Conversely, let us now assume that $x$ is strongly feasible and also optimal for program~\eqref{eq:strOpt:c}. Strong feasibility of $x$ implies that $\lb{d} = \ub{d}$ (see Proposition~\ref{prop:strFeas:sol}), so the demand vector in the problem must be fixed. Thus, we can write $\ub{d}$ in the considered scenario \eqref{eq:strOpt:c}. It remains to show that if $x$ is optimal for the scenario $(\ub{c}, \ub{s}, \ub{d})$, then it is also optimal for any other supply vector $s \in \ib{s}$ and the corresponding scenario $(\ub{c}, {s}, \ub{d})$ and therefore it is strongly optimal.
	
	Suppose that $x$ is optimal for the scenario $(\ub{c}, \ub{s}, \ub{d})$, i.e. it is an optimal solution of program~\eqref{eq:strOpt:c}. Assume for contradiction that there exists a supply vector $s \in \ib{s}$ such that $x$ is not optimal for $(\ub{c}, s, \ub{d})$. By strong feasibility of $x$, we also have $x \in \mathcal{F}(s, \ub{d})$. Let $x^*$ denote an optimal solution for $(\ub{c}, s, \ub{d})$. Since $x$ is not optimal, it has a worse objective value, i.e.
	\begin{equation}\label{eq:obj:proof}
		\sum_{i \in I} \sum_{j \in J} \ub{c}_{ij} x^*_{ij} < \sum_{i \in I} \sum_{j \in J} \ub{c}_{ij} x_{ij}.
	\end{equation}
	However, $x^*$ is also feasible for the scenario $(\ub{c}, \ub{s}, \ub{d})$, since we have \[ \sum_{j \in J} x^*_{ij} \le s_i \le \ub{s}_i, \qquad \forall i \in I.\]
	Therefore, we have found a solution $x^*$, which is feasible for $(\ub{c}, \ub{s}, \ub{d})$ and by \eqref{eq:obj:proof} it has a strictly better objective value than $x$. This contradicts the assumption that $x$ is an optimal solution of program~\eqref{eq:strOpt:c}.
\end{proof}

In general, checking strong optimality of a solution is a co-NP-hard problem for interval linear programs, even if we only have interval uncertainty in the objective function. Proposition~\ref{prop:strOpt:sol:poly} presents a necessary and sufficient condition for strong optimality, which is known for interval linear programs (see \cite{Hladik:2017:strongopt}).

\begin{proposition}[from \cite{Hladik:2017:strongopt}]\label{prop:strOpt:sol:poly}
		A given $x \in \RR^{m \times n}$ is a strongly optimal solution of \eqref{eq:itp:unbalanced} if and only if it is strongly feasible and the following interval system is not weakly feasible:
		\begin{equation}
			\begin{alignedat}{3}
			\sum_{i \in I} \sum_{j \in J} [\lb{c}_{ij}, \ub{c}_{ij}] u_{ij} &\le -1,\\
			\sum_{i \in I} u_{ij} &= 0, &&\forall j \in J,\\
			\sum_{j \in J} u_{ij} &\le 0, &&\forall i \in I^*,\\
			u_{ij} &\ge 0, &\qquad&\forall (i,j) \in K^*,
			\end{alignedat}
		\end{equation}
	where $I^* = \{i \in I : \sum_{j \in J} x_{ij} = \ub{s}_i\}$ and $K^* = \{(i,j) \in I \times J : x_{ij} = 0 \}$.
\end{proposition}

Weak and strong optimality of the interval transportation problem, i.e. checking whether some/each scenario has an optimal solution, can be reduced to the problem of checking weak and strong feasibility. This is possible thanks to the fact that the transportation problem with non-negative costs and non-negative variables always has an objective value bounded below by $0$.
\begin{proposition}
	Problem \eqref{eq:itp:unbalanced} is weakly optimal if and only if it is weakly feasible.
	Furthermore, problem \eqref{eq:itp:unbalanced} is strongly optimal if and only if it is strongly feasible.
\end{proposition}

Note that while weak optimality of an interval program is equivalent to the existence of a weakly optimal solution, an analogous property does not hold for strong optimality. An interval program is strongly optimal if each scenario has an optimal solution, but the optimal solution may be different for each scenario and a common strongly optimal solution might not exist. 

\subsection{Optimal Value Range}
Finally, we address the problem of computing the best and the worst optimal value of an interval transportation problem. Computing the best optimal value~$\lb{f}$ only requires solving a single linear program. As shown in Proposition~\ref{prop:BestVal}, the linear program minimizes the best objective $\lb{c}$ over the set of all weakly feasible solutions.
\begin{proposition}[\cite{Liu:2003:ITP}]\label{prop:BestVal}
	The best optimal value $\lb{f}$ of \eqref{eq:itp:unbalanced} can be computed as the optimal value of the linear program
	\begin{equation}
		\begin{alignedat}{3}
			&\textnormal{minimize } &&\sum_{i \in I} \sum_{j \in J} \lb{c}_{ij} x_{ij} \\
			&\textnormal{subject to } &&\eqref{eq:prop:weakFeas}.
		\end{alignedat}
	\end{equation}			
\end{proposition}
The worst (possibly infinite) value $\ub{f}$ can also be computed using linear programming. If there is an infeasible scenario of the ITP, then by definition we have $\ub{f} = \infty$. If the ITP is strongly feasible, it can be again shown that solving a~single scenario is sufficient.

\begin{proposition}[\cite{Xie:2017:ITP}]\label{prop:WorstVal}
	If the problem \eqref{eq:itp:unbalanced} is not strongly feasible, then we have $\ub{f} = \infty$. Otherwise, $\ub{f} = \ubf$ can be computed as the optimal value of the linear program
		\begin{equation}
		\begin{alignedat}{3}
			&\textnormal{minimize } &&\sum_{i \in I} \sum_{j \in J} \ub{c}_{ij} x_{ij} \\
			&\textnormal{subject to } &&\eqref{eq:prop:strFeas}.
		\end{alignedat}
	\end{equation}	
\end{proposition}

Methods for computing the worst finite optimal value $\ubf$ for interval transportation problems that are not strongly feasible will be further discussed in Section~\ref{sec:worstFinVal}. Regarding the complexity of the corresponding decision problem, it was proved to be NP-hard for linear programs with interval right-hand-side vectors, but it is decidable in polynomial time for problems, in which intervals are only present in the objective vector. By the results of \cite{Hladik:2018:FinWorstCase}, we obtain the characterization in Proposition~\ref{prop:worstFinVal} for transportation problems with interval costs.
\begin{proposition}[from \cite{Hladik:2018:FinWorstCase}]\label{prop:worstFinVal}
	If $\lb{d} = \ub{d}$ and $\lb{s} = \ub{s}$, then $\ubf$ of \eqref{eq:itp:unbalanced} can be computed as the optimal value of the linear program
	\begin{equation}
		\begin{alignedat}{3}
			&\textnormal{maximize } &\sum_{i \in I} \ub{s}_i u_i + \sum_{j \in J} \ub{d}_j v_j\\
			&\textnormal{subject to } &\textnormal{constraints }\eqref{eq:prop:weakFeas},\\
			& &u_i + v_j \le \ub{c}_{ij}, &\qquad\forall i \in I, j \in J,\\
			& &u_i \le 0,\ \, &\qquad\forall i \in I.
		\end{alignedat}
	\end{equation}	
\end{proposition}
Specifically for the interval transportation problems, it was recently proved by \cite{Hoppmann:2021:MinCostFlow} that computing $\ubf$ is NP-hard for the \eqref{eq:itp:balanced} formulation. The complexity of \eqref{eq:itp:unbalanced} for problems that are not strongly feasible remains open.

\section{Computing the Worst Finite Optimal Value}\label{sec:worstFinVal}
Computing the worst finite optimal value $\ubf$ of an interval transportation problem is a challenging task, in general. The methods for computing $\ubf$ available in the literature are mainly designed to find an approximation of the exact value, which can be NP-hard to compute at least for some classes of instances.

\subsection{Overview of Algorithms for Computing $\ubf$}
\cite{Juman:2014:ITP} proposed a heuristic solution technique to find a~bound on $\ubf$ for an extended model with inventory costs. \cite{Xie:2017:ITP} designed a permutation heuristic genetic algorithm and proved that $\ubf$ can be computed exactly in polynomial time for strongly feasible instances of \eqref{eq:itp:unbalanced}, as stated in Proposition~\ref{prop:WorstVal}. Finally, \cite{Cerulli:2017:BestWorstValues} introduced a heuristic algorithm for computing the worst finite optimal value based on an iterative local search method. 

Apart from strongly feasible instances of \eqref{eq:itp:unbalanced}, another interesting class of interval transportation problems that was investigated are problems with a~cost matrix immune against the transportation paradox. \cite{DAmbrosio:2020:optimalvaluerange} derived stronger theoretical properties of the worst optimal value on the class of immune problems and proposed a more powerful algorithm to approximate the value. Later, \cite{Carrabs:2021:IITP} designed a heuristic approach based on some polyhedral properties of the problem and introduced a mixed-integer linear programming formulation to compute the value exactly.

Here, we aim to design an exact method for computing the worst finite optimal value $\ubf$ for general interval transportation problems without any further assumptions. We present the method for the \eqref{eq:itp:unbalanced} formulation, however, it is easy to adapt it to the equation-constrained \eqref{eq:itp:balanced} formulation and also to other classes of interval linear programs. Another exact approach to computing $\ubf$ was suggested by \cite{Liu:2003:ITP}, who derived a linearly constrained bilinear program based on duality. Unfortunately, finding a globally optimal solution of the nonlinear program seems to be quite difficult even for very small problems (as was also illustrated by \cite{Juman:2014:ITP}).

\subsection{Decomposition by Complementary Slackness}
In~\cite{Garajova:2020:BestWorstITP}, we proposed a method for computing $\ubf$, based on a~decomposition by complementary slackness for describing the set of all weakly optimal solutions of~\eqref{eq:itp:unbalanced}. Using the decomposition, we can compute $\ubf$ by optimizing the worst objective with the highest cost values
\[ \textnormal{maximize }\sum_{i \in I} \sum_{j \in J} \ub{c}_{ij} x_{ij}\]
over the sets of optimal solutions described by linear systems in the form
\begin{subequations}\label{eq:worst:decomp}
	\begin{align}
		\sum_{j \in J} x_{ij} \le \ub{s}_i, &  \quad& & \forall i \in I, \label{eq:decomp:1}\\
		\lb{d}_j \le \sum_{i \in I} x_{ij} \le \ub{d}_j,	& \quad& & \forall j \in J, \label{eq:decomp:2}\\
		u_i + v_j \le \ub{c}_{ij}, &\quad x_{ij} = 0, 		&\quad &\forall	(i,j) \in K, \label{eq:decomp:3}\\
		u_i + v_j = \ub{c}_{ij},  &\quad x_{ij} \ge 0, 	& \quad &\forall	(i,j) \notin K, \label{eq:decomp:4}\\
		\sum_{j \in J} x_{ij} \ge \lb{s}_i,\ & \quad\ u_i \le 0, \ & &\forall i \in L, \label{eq:decomp:5}\\[-10pt]
		& \quad\ u_i = 0, \ &\qquad&\forall i \notin L, \label{eq:decomp:6}
	\end{align}
\end{subequations}
for all subsets $K \subseteq I \times J$ and $L \subseteq I$ (in fact, the number of programs can be reduced by considering only basic optimal solutions). If we denote by $f_{KL}$ the optimal value of such linear program for particular subsets $K$ and $L$, then the worst finite optimal value can be computed as
\[
\ubf = \max\, \{ f_{KL} : K \subseteq I \times J, \, L \subseteq I \}.
\]
Constraints \eqref{eq:decomp:1}--\eqref{eq:decomp:2} ensure primal feasibility of the solution $x$, the remaining constraints \eqref{eq:decomp:3}--\eqref{eq:decomp:6} then ensure dual feasibility of the solution $(u,v)$ and complementary slackness, which replaces the strong duality condition~\eqref{eq:itp:strongdual}.

However, instead of explicitly decomposing the problem into an exponential number of linear programs, we can also formulate the complementary slackness constraints as implications, namely
\begin{align*}
	x_{ij} > 0 \ \Rightarrow\  u_i + v_j = \ub{c}_{ij}, \quad\text{ and, }\quad
	\sum_{j \in J} x_{ij} < \lb{s}_i \ \Rightarrow\  u_i = 0.
\end{align*}
These implications can also be modeled as indicator constraints, which can be handled seamlessly by nowadays solvers.

Then, we can compute $\ubf$ by solving the optimization problem
\begin{subequations}\label{eq:worst:decomp:indicator}
	\begin{align}
		\textnormal{maximize}&&\sum_{i \in I} \sum_{j \in J} \ub{c}_{ij} x_{ij}\notag\\ 
		\textnormal{subject to}&&\sum_{j \in J} x_{ij} &\le \ub{s}_i, &  & \forall i \in I,\label{eq:decomp:ind:1}\\
		&&\lb{d}_j \le \sum_{i \in I} x_{ij} &\le \ub{d}_j,	& & \forall j \in J,\label{eq:decomp:ind:2}\\[-5pt]
		&&x_{ij} &\ge 0, & & \forall i \in I, j \in J,\label{eq:decomp:ind:3}\\
		&&u_i + v_j &\le \ub{c}_{ij}, 		& &\forall i \in I, j \in J,\label{eq:decomp:ind:4}\\
		&&u_i &\le 0, 	&  &\forall	i \in I,\label{eq:decomp:ind:5}\\
		&&x_{ij} > 0 \ &\Rightarrow\  u_i + v_j = \ub{c}_{ij}, &  & \forall i \in I, j \in J,\label{eq:decomp:ind:6}\\
		&&\sum_{j \in J} x_{ij} < \lb{s}_i \ &\Rightarrow\  u_i = 0, &  & \forall i \in I,\label{eq:decomp:ind:7}
	\end{align}
\end{subequations}
where the constraints \eqref{eq:decomp:ind:1}--\eqref{eq:decomp:ind:3} ensure primal feasibility of $x$, the constraints \eqref{eq:decomp:ind:4}--\eqref{eq:decomp:ind:5} ensure dual feasibility of $(u,v)$ and the remaining constraints \eqref{eq:decomp:ind:6}--\eqref{eq:decomp:ind:7} ensure complementary slackness.

\subsection{Initial Feasible Solution}
We also provide the solver with an initial feasible solution of \eqref{eq:worst:decomp}, i.e. a weakly optimal solution of the ITP. Such a feasible solution is encoded via variables $x, u, v$ and by index sets $K$ and $L$. To construct the initial solution we simply select an initial scenario $(s,d)$, optimize \eqref{eq:itp:unbalanced} for this scenario (this yields $x$ as primal variables and $u,v$ as dual ones) and build the index sets such that the conditions of \eqref{eq:worst:decomp} are met. 

We can assume that the problem instance is weakly feasible, otherwise there is no worst finite optimal value. Note that weak feasibility can be easily checked by verifying $\sum_{i=1}^m \ub{s}_i \ge \sum_{j=1}^n \lb{d}_j$. Among all feasible scenarios, we select the one with the maximal amount of transported goods. Let us set  \[ g = \min \left\{\sum_{i=1}^m \ub{s}_i,\,\sum_{j=1}^n \ub{d}_j \right\}, \]
which represents the maximal demand that can be satisfied by the available supply.
Then, we consecutively set for all $i \in \{1,\ldots,m\}$ and for all $j \in \{1,\ldots,n\}$ the values of supply and demand as
\begin{align*}
    s_i &= \max \left\{ \lb{s}_i,\  \min \left\{ \ub{s}_i, g- \sum_{k=1}^{i-1} s_k - \sum_{k=i+1}^m \lb{s}_k\right\}\right\},\\
    d_j &= \min\left\{\ub{d}_j,\  g-\sum_{k=1}^{j-1} d_k - \sum_{k=j+1}^n \lb{d}_k\right\}.
\end{align*}
In other words, we find the largest index $j^*$, such that when we set the demand at the first $({j^*}{}\!-\!{}1)$ nodes to the highest value $\ub{d}_j$, there is sufficient supply to cover the lowest possible demand $\lb{d}_j$ at the remaining nodes. Then, we set
\begin{align*}
&d_j = \ub{d}_j \quad \text{ for } j \in \{1, \dots, j^*-1\},\\
&d_j = \lb{d}_j \quad \text{ for } j \in \{j^*+1, \dots, n\},\\
&d_{j^*} = g - \sum_{k=1}^{j^*\!-1} \ub{d}_k - \sum_{k=j^*\!+1}^n \lb{d}_k.
\end{align*}
Analogously, we try to set the supply at the nodes to exactly satisfy the demand~$g$, however, if $\sum_{k=1}^m \lb{s}_k > \sum_{k=1}^n \ub{d}_k$, then the total supply is set to a~greater value than $g$.

\section{Computational Experiments}\label{sec:experiment}
We started from our previous experiments presented in \cite{garajova:2021:ExactMethodWorst}. We broaden the experiments and test the improved implementation of the exact method. One of the changes is that we supply the solver with an initial solution to reduce the time to obtain a high-quality solution.

\subsection{Implementation of the Exact Method}
Our method was implemented in Python 3.9. We used Gurobi 9.5 to solve the corresponding mathematical models. It turned out that the most efficient way to model the decomposition in \eqref{eq:decomp:3}-\eqref{eq:decomp:6}, i.e. the complementary slackness, is to use indicator constraints, which are natively supported by Gurobi. We also tested the explicit decomposition into linear programs and a mixed-integer programming reformulation of the decomposition without any special constraints, however, indicator constraints outperformed both of them.
The Python code is available from authors on request.

The experiment was carried out on a computer with a 16\,GB RAM and an AMD Ryzen 7 5800X processor. 

\subsection{Instances}
We used two datasets from \cite{Xie:2017:ITP} comprising $65$ instances of interval transportation problems of different sizes. In the following descriptions, the sizes denote the number of supply centers $m$ and the number of destinations~$n$, rather than the total number of constraints ($n+m$) and variables ($n \times m$) in the resulting model. 

The first smaller dataset is composed of five interval transportation problems of five different sizes $3\times 5$, $4\times 6$, $5 \times 10$, $10 \times 10$ and $20 \times 20$ taken from~\citet[Tables 2--11]{Xie:2017:ITP}. In all instances, the supply and demand interval vectors $\ib{s}$ and $\ib{d}$ are chosen so that $\ub{s} = 2\lb{s}$ and $\ub{d} = 2\lb{d}$.

The second dataset consists of $60$ benchmark instances of six different sizes, which were used by \citet[Tables 13--18]{Xie:2017:ITP} and \citet[Table 1]{Cerulli:2017:BestWorstValues}. The instances in dataset 2 have the same supply and demand intervals as the instances of corresponding size from dataset 1, while the unit transportation costs were generated randomly.

\subsection{Results and Discussion: Dataset 1}
In the first experiment, we tested the efficiency and quality of the proposed method (denoted in the results as MILP) against the heuristic algorithms designed by \cite{Liu:2003:ITP}, \cite{Juman:2014:ITP} and \cite{Xie:2017:ITP}. The different methods for computing the worst finite optimal value $\ubf$ were compared on the smaller dataset of five interval transportation problems of five different sizes up to $20 \times 20$. Results of the experiment are shown in Table~\ref{tab:2}.

For the benchmark instances with up to $10$ supply centers and up to $10$ destinations, the solver was able to find the exact optimal solution for the MILP method almost instantly. In all of these instances, the solution found was strictly better (with a higher objective value) than the solutions provided by the fast algorithms of \cite{Liu:2003:ITP} and \cite{Juman:2014:ITP}. The heuristic method by \cite{Xie:2017:ITP} was able to compute the same solutions as MILP, however, the required computation time was longer. Thus, for the smaller instances, the MILP method can be used to compute the value $\ubf$ exactly and very quickly.

Due to the complexity of the problem, it can be expected that the price for a~high-quality solution rises when the number of supply centers and destinations is increased. For the $20\times 20$ instance, the solver was unable to find a provably optimal solution of MILP (i.e., the exact value of $\ubf$) within the time limit of $25$ minutes. Nevertheless, after only $5$ seconds of computations, the solver returned a solution with the value $9425$, which is already higher than the other solutions found by the competing heuristic algorithms.



\defcitealias{Juman:2014:ITP}{Juman \& H.}
\begin{table}[t]
	\caption{The obtained worst finite optimal values and running times (in seconds) of the four considered methods on instances of $5$ different sizes from~\cite{Xie:2017:ITP}.}\label{tab:2}
	\begin{tabular*}{\textwidth}{@{\quad}c>{\raggedleft}p{0.9cm}@{\ \ }p{0.9cm}>{\raggedleft}p{1.1cm}@{\ \ }p{1.1cm}>{\raggedleft}p{0.8cm}@{\ \ }p{0.8cm}>{\raggedleft}p{1.1cm}@{\ \ }p{1.1cm}}				
		\hline
		$m\times n$&	\multicolumn{2}{c}{MILP} &	\multicolumn{2}{c}{\cite{Xie:2017:ITP}} & \multicolumn{2}{c}{\cite{Liu:2003:ITP}} & \multicolumn{2}{c}{\citetalias{Juman:2014:ITP} \citeyearpar{Juman:2014:ITP}}\\ \hline
		$3\times5$					&	9555 &	(0.05\,s)	 &	9555&	\hphantom{0}(10.33\,s)&	8615&	(0.04\,s)&	8615&	(0.69\,s)\\
		$4\times6$					&	15060&	(0.05\,s)	 &	15060&	\hphantom{0}(13.07\,s)&	13235&	(0.03\,s)&	13235&	(6.84\,s)\\
		$\hphantom{0}5\times10$		&	13175&	(0.09\,s)	 &	13175&	\hphantom{0}(20.03\,s)&	9420&	(0.03\,s)&	9330&	(1.05\,s)\\
	    $10\times10$				&	26350&	(0.49\,s)	 &	26350&	\hphantom{0}(26.92\,s)&	18840&	(0.04\,s)&	18600&	(1.05\,s)\\
		$20\times20$				&	9425&	(1500\,s)\footnotemark[1]&	9320&	(362.52\,s)&	9070&	(0.03\,s)&	8210&	(4.08\,s)\\ \hline
	\end{tabular*}	
	\footnotetext[1]{The computation was terminated after $25$ minutes. The solution $9425$ was found after $5$\,s, but the solver was not able to decide its optimality.}
\end{table}	

\subsection{Results and Discussion: Dataset 2}
In the second experiment, we compared our method to the algorithm by \cite{Xie:2017:ITP}, which provided the best-quality solution from the heuristic methods tested in the former experiment, and against the newer local search algorithm by \cite{Cerulli:2017:BestWorstValues}. Here, the computation of $\ubf$ was performed on the larger set of $60$ benchmark instances. The results of the second experiment are shown in Table~\ref{tab:3}. The values computed by the other two algorithms are taken from~\cite{Xie:2017:ITP} and \cite{Cerulli:2017:BestWorstValues}.

The running time of our exact method was negligible for instances up to size $10 \times 10$. Hence, the obtained solutions for these instances were proven to be optimal. For instances of size $20 \times 20$, we set the time limit to 1500 seconds. It turns out that the exact MILP formulation gives the best overall results (even if we lower the time limit significantly), followed by the heuristic algorithm of \cite{Cerulli:2017:BestWorstValues}. The boldly-typeset objective values are those that are best for a particular instance (except for instances where all the obtained values are the same). The heuristic from \cite{Cerulli:2017:BestWorstValues} was able to find a better solution for one $20 \times 20$ instance (instance id 2). On the other hand, there are multiple instances for which the MILP algorithm dominates.

Table~\ref{tab:3} also gives evidence about occurrence of the \emph{transportation paradox} in the instances. This is illustrated by the best solutions returned by the MILP method. For every instance, it holds that $\sum_{i=1}^m \ub{s}_i = \sum_{j=1}^n \ub{d}_j$. Also, all instances of the same size have the same value of $\sum_{j=1}^n \ub{d}_j$. Table \ref{tab:3} provides this number in the very first column for each group of instances. Then, the column labeled by $\sum d$ gives evidence of the amount of transported goods in the optimal scenario (or, in case of $20 \times 20$ instances, the best scenario found). 
It is worth noting that the case $\sum d < \sum \ub{d}$ occurs very often. So, the transportation paradox is a very common phenomenon in this class of instances. 


The last column of Table \ref{tab:3} (labeled time) gives the elapsed time at which the best solution for an instance is found. This is interesting only for $20 \times 20$ instances, as for the smaller instances an optimal solution was found. We can see that the maximum time to find the best solution is quite short (151 seconds) compared to the available computation time restricted by the time limit of 1500 seconds.

\begin{sidewaystable}
	\footnotesize
	\sidewaystablefn
	\newcommand{\STAB}[1]{\begin{tabular}{@{}c@{}}#1\end{tabular}}	
	\caption{Columns Xie, Cerulli and MILP: the obtained worst finite optimal value of the three considered methods on $60$~instances of the interval transportation problem of $6$ different sizes from~\cite{Xie:2017:ITP}. Column $\sum d$ refers to the total amount of goods transferred in the worst scenario. The column time shows the time required to find the reported solution (in seconds).}\label{tab:3}
	\begin{center}
	\begin{tabular}{@{\hskip20pt}ccccccc@{\hskip10pt}|@{\hskip20pt}cccccccc@{\hskip20pt}}								
		\hline
		&id&	Xie&	Cerulli&	MILP&$\sum d$&	&&id&Xie&	Cerulli&	MILP&$\sum d$& time	(s) \\
		\hline
		\multirow{10}{*}{\STAB{\rotatebox[origin=c]{90}{\begin{minipage}{2cm}size: $2\times3$\newline $\sum \ub{d}$: 270 \end{minipage}}}}&1&	22800&	22800&	22800& 270& &	\multirow{10}{*}{\STAB{\rotatebox[origin=c]{90}{\begin{minipage}{2cm}size: $5\times10$\newline $\sum \ub{d}$: 1090 \end{minipage}}}}&1&	25760&\bf	25810&\bf	25810 &910 &
		\\
		&2&	27390&	27390&	27390&	210&&&2&	24980&	\bf 25055&	\bf 25055&795&	\\
		&3&	27390&	27390&	27390&	210&&&3&	19635&	19635&	19635&865&	\\
		&4&	27210&	27210&	27210&	240&&&4&	30260&	30260&	30260&955&	\\
		&5&	18570&	18570&	18570&	270&&&5&	23590&	23590&	23590&865&	\\
		&6&	30900&	30900&	30900&	270&&&6&	23075&	23075&	23075&860&	\\
		&7&	22020&	22020&	22020&	210&&&7&	22375&	\bf 22740&\bf	22740&930&	\\
		&8&	18450&	18450&	18450&	210&&&8&	30000&	30000&	30000&955&	\\
		&9&	21450&	21450&	21450&	210&&&9&	24675&	24675&	24675&855&	\\
		&10&	14130&	14130&	14130&	210&&&10&	39985&	39985&	39985&975&	\\
		\hline
		\multirow{10}{*}{\STAB{\rotatebox[origin=c]{90}{\begin{minipage}{2cm}size: $3\times5$\newline $\sum \ub{d}$: 470 \end{minipage}}}}&1&	16410&	16410&	16410&410&&	\multirow{10}{*}{\STAB{\rotatebox[origin=c]{90}{\begin{minipage}{2cm}size: $10\times10$\newline $\sum \ub{d}$: 2180 \end{minipage}}}}&1&	36180&	36180&	36180&1760&
		\\
		&2&	14820&	14820&	14820&410&&	&2&	43260&	43260&	43260&1845&	\\
		&3&	20650&	20650&	20650&445&&	&3&	38915&	38915&	38915&1850&	\\
		&4&	12940&	12940&	12940&410&&	&4&	38845&	38905&	\bf 38915&1945&	\\
		&5&	16650&	16650&	16650&470&&	&5&	50150&	50150&	50150&1809&	\\
		&6&	16540&	16540&	16540&410&&	&6&	29885&	29885&	29885&1760&	\\
		&7&	10195&	10195&	10195&395&&	&7&	44100&	\bf 44145&	\bf 44145&1680&	\\
		&8&	\bf 13360&	12620&	\bf 13360&410&&	&8&	41950&	41950&	\bf 42420&1710&	\\
		&9&	11010&	11010&	11010&395&&	&9&	\bf 37465&	37180&	\bf 37465&1940& \\
		&10&	12915&	12915&	12915&445&&	&10&	48920&	48920&	48920&1810&	\\
		\hline
		\multirow{10}{*}{\STAB{\rotatebox[origin=c]{90}{\begin{minipage}{2cm}size: $4\times6$\newline $\sum \ub{d}$: 630 \end{minipage}}}}&1&	27125&	27125&	27125&630 &&	\multirow{10}{*}{\STAB{\rotatebox[origin=c]{90}{\begin{minipage}{2cm}size: $20\times20$\newline $\sum \ub{d}$: 3540 \end{minipage}}}}&1&	\phantom{0}9405&	\phantom{0}9405&	\bf\phantom{0}9425\footnotemark[1]&2840&4&	\\
		&2&	20635&	20635&	20635&615&&	&2&	\phantom{0}9015&	\bf\phantom{0}9140&	\phantom{0}9115\footnotemark[1]&2910&41	\\
		&3&	23615&	23615&	23615&585&&	&3&	\phantom{0}9335&	\phantom{0}9405&	\bf\phantom{0}9425\footnotemark[1]&2840&6	\\
		&4&	22375&	22375&	22375&645&&	&4&	\phantom{0}8930&	\bf\phantom{0}9130&	\bf\phantom{0}9130\footnotemark[1]&2720&151	\\
		&5&	19500&	19500&	19500&645&&	&5&	\phantom{0}9275&	\bf\phantom{0}9420&	\bf \phantom{0}9420\footnotemark[1]&2795&2	\\
		&6&	11380&	11380&	11380&535&&	&6&	10220&	\bf 10320&	\bf 10320\footnotemark[1]&2950&24	\\
		&7&	17245&	17245&	17245&640&&	&7&	\phantom{0}8685&	\phantom{0}8630&	\bf \phantom{0}8700\footnotemark[1]&3090&6	\\
		&8&	24180&	24180&	24180&660&&	&8&	\phantom{0}9260&	\phantom{0}9260&	\phantom{0}9260\footnotemark[1]&2760&15	\\
		&9&	24060&	24060&	24060&600&&	&9&	\phantom{0}9885&	\phantom{0}9885&	\phantom{0}9885\footnotemark[1]&2965&8	\\
		&10&	22825&	22825&	22825&640&&	&10&	\phantom{0}9225&	\phantom{0}9220&	\bf\phantom{0}9290\footnotemark[1]&2985&1	\\
		\hline
	\end{tabular}				
	\end{center}
	\footnotesize
	\footnotetext[1]{Possibly not optimal. The computation was terminated after $25$ minutes.}  
\end{sidewaystable}

\begin{figure}
	\centering
	\includegraphics[page=1]{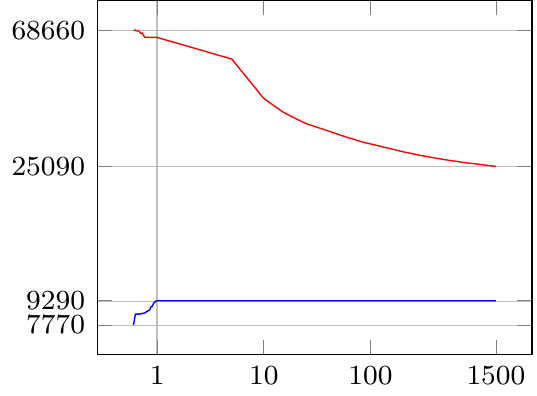}\hfill
	\includegraphics[page=2]{graf1.pdf}\vskip10pt
	\includegraphics[page=3]{graf1.pdf}\hfill
	\includegraphics[page=4]{graf1.pdf}
	\caption{Objective value (vertical axis) vs. time (horizontal axis) for $20\times 20$ instances with ids 10 (upper left), 1 (upper right), 8 (lower left) and 4 (lower right). Both axes are logarithmic. Red curve: upper bound (from Gurobi), blue curve: best solution found. Vertical leads denote the time of finding of the best solution found. Usually, the blue curves are flat as good feasible solutions are found quickly, the red curves decrease quite slowly. The hard part of the computations is to prove that a solution is optimal.}
	\label{fig:4:graphs}
\end{figure}

Figure \ref{fig:4:graphs} depicts the process of computation for 4 instances of size $20 \times 20$. It shows how the upper bound (which is computed by Gurobi) for the objective value approaches the best solution found during the computation. It can be seen that the upper bound follows almost the same curve for each of the instances. The figures indicate that a large portion of the computation time might be devoted to proving optimality of a solution that was found. It is a tempting question whether the space of subproblems can be significantly reduced to improve the tightness of upper bound. On the other hand, we can see that a good feasible solution is usually found quickly.

\section{Conclusion}\label{sec:concl}
Transportation problems belong to the most important models of operations research and play a crucial role in logistics and supply-chain management. While the traditional setting assumes that all data in the model are known exactly, this is not always the case in real-world situations, where the customer demand, supply quantities and transportation costs may change over time.
We considered the model of a transportation problem under interval uncertainty, in which the supply, demand and costs can be independently perturbed within some given lower and upper bounds. 

We studied the theoretical properties of interval transportation problems, including the problem of checking weak and strong feasibility or optimality of a given solution, i.e. deciding whether a given solution is feasible or optimal for some/all possible scenarios of the problem. We proved that weak optimality of a solution can be checked in polynomial time by solving a linear system for a suitable scenario. We also derived an analogous condition for checking strong optimality of a solution for problems with exact costs. Furthermore, we studied weak and strong feasibility and optimality of the interval transportation problem itself and surveyed the results on computing the optimal value range.

Knowing how the optimal values and optimal solutions change in the best or the worst case when the model is affected by uncertainty can aid in the decision-making process to modify the transportation plan or production quantities in order to reduce the total transportation costs. Here, we addressed the problem of computing the worst finite optimal value, which was recently proved to be NP-hard for equation-constrained interval transportation problems. We improved a complementarity-based method for computing the worst value exactly and included finding an initial feasible solution to speed up the computation. 

Although the method requires exponential time in the worst case, the conducted numerical experiments indicate that it can be competitive with the state-of-the-art heuristic algorithms for solving the problem on general instances of the interval transportation problem. Moreover, even if the instance size is too large so that the exact computation becomes intractable, the method can still be used with a limited time to provide a good approximation of the worst finite optimal value. This indicates that the method could also serve as a basis for a faster heuristic for approximating the worst optimal value in a~reasonable time, which would make it suitable for use in practical applications. 

\bmhead{Acknowledgments}
The authors wish to thank the authors of \cite{Xie:2017:ITP} and \cite{Cerulli:2017:BestWorstValues} for sharing the input data of the benchmark instances used in the computational experiments.

\section*{Declarations}

\paragraph{\rm\textbf{Funding} E. Garajová and M. Rada were supported by the Czech Science Foundation under Grant P403-20-17529S.}

\paragraph{\rm\textbf{Conflict of interest} The authors have no conflicts of interest to declare that are relevant to the content of this article.}

\paragraph{\rm\textbf{Ethical approval / Informed consent} This article does not contain any studies with human participants or animals performed by any of the authors.}

\paragraph{\rm\textbf{Authorship Contributions} All authors contributed equally to the preparation of the manuscript.}

\paragraph{\rm\textbf{Data Availability} The data that support the findings of this study are available from the corresponding author upon request.}

\bibliography{sor21}


\end{document}